\documentclass[12pt,leqno]{article}
\usepackage{amsmath,amssymb,bbm,array}
\usepackage{amsthm}
\usepackage[all,cmtip]{xy}
\def\bm#1{\mathbbm{#1}}
\def\fn#1{\mathop{{\rm #1}\vphantom{\dim}}}

\def\hs#1{\hspace*{#1ex}}

\def\vds{\hbox{\hbox to 2\arraycolsep{\hss\vbox{\vbox to 1.3ex{ 
\vss\hbox{.}\vspace{.33ex}\hbox{.}\vspace{.33ex}\hbox{.}\vspace{.33ex}\hbox{.}\vspace{.33ex}\hbox{.}\vss}}\hss}}}
\def\hds#1{\hdotsfor[-.1]{#1}}  

\makeatletter
\renewcommand{\section}{\@startsection{section}{1}{0mm}{12mm}{5mm}{\raggedright\bf\large}}

\def\@citex[#1]#2{\if@filesw\immediate\write\@auxout{\string\citation{#2}}\fi
  \def\@citea{}\@cite{\@for\@citeb:=#2\do
    {\@citea\def\@citea{\@citesep}\@ifundefined
       {b@\@citeb}{{\bf ?}\@warning
       {Citation `\@citeb' on page \thepage \space undefined}}%
{\csname b@\@citeb\endcsname}}}{#1}}
\def\@citesep{; }
\makeatother

\newtheoremstyle{Kang}{}{}{\itshape}{}{\bf}{}{.5em}{}
\theoremstyle{Kang}
\newtheorem{theorem}{Theorem}[section]
\renewcommand{\thetheorem}{\arabic{section}.\arabic{theorem}}

\newtheoremstyle{Kremark}{}{}{}{}{\bf}{}{.5em}{}
\theoremstyle{Kremark}
\newtheorem{defn}[theorem]{Definition}
\newtheorem*{remark}{Remark.}
\newtheorem{other}{}

\allowdisplaybreaks[1]  
\numberwithin{equation}{section}

\title{THE RATIONALITY PROBLEM FOR FINITE SUBGROUPS OF $GL_4(\bm{Q})$}
\author{Ming-chang Kang$^{a,1}$ and Jian Zhou$^{b,1,2}$ \\[3mm]
\begin{minipage}{16cm} \begin{description} \itemsep=-1pt
\item[] $^{a}$Department of Mathematics and Taida Institute of
Mathematical\\ Sciences, National Taiwan University, Taipei
\item[] $^{b}$School of Mathematical Sciences, Peking University,
Beijing
\end{description} \end{minipage}}
\date{}

\begin{document}

\maketitle

\footnote{\hspace*{-7.5mm}
E-mail addresses: kang@math.ntu.edu.tw, zhjn@math.pku.edu.cn.}
\footnote{\hspace*{-6mm}$^1\,$Both authors were partially
supported by National Center for Theoretic Sciences (Taipei
Office). } \footnote{\hspace*{-6mm}$^2\,$The work of this paper
was finished when the second-named author visited National Taiwan
University under the support by National Center for Theoretic
Sciences (Taipei Office).}

\begin{abstract}
{\bf Abstract.} Let $G$ be a finite subgroup of $GL_4(\bm{Q})$.
The group $G$ induces an action on $\bm{Q}(x_1,x_2,x_3,x_4)$, the
rational function field of four variables over $\bm{Q}$. Theorem.
The fixed subfield
$\bm{Q}(x_1,x_2,x_3,x_4)^G:=\{f\in\bm{Q}(x_1,x_2,x_3,x_4):\sigma
\cdot f=f$ for any $\sigma\in G\}$ is rational (i.e.\ purely
transcendental) over $\bm{Q}$, except for two groups which are
images of faithful representations of $C_8$ and $C_3\rtimes C_8$
into $GL_4(\bm{Q})$ (both fixed fields for these two exceptional
cases are not rational over $\bm{Q}$). There are precisely 227
such groups in $GL_4(\bm{Q})$ up to conjugation; the answers to
the rationality problem for most of them were proved by Kitayama
and Yamasaki \cite{KY} except for four cases. We solve these four
cases left unsettled by Kitayama and Yamasaki; thus the whole
problem is solved completely.
\end{abstract}

\section{Introduction}

Let $G$ be a finite subgroup of $GL_n(\bm{Q})$ and
$\bm{Q}(x_1,x_2,\ldots,x_n)$ be the rational function field of $n$
variables over $\bm{Q}$. The group $G$ induces an action on
$\bm{Q}(x_1,\ldots,x_n)$ by $\bm{Q}$-automorphisms defined as
follows: For any $\sigma=(a_{ij})_{1\le i,j \le n}\in
GL_n(\bm{Q})$, for any $1\le j\le n$, define $\sigma\cdot
x_j=\sum_{1\le i\le n} a_{ij}x_i$. In this article we will like to
know whether the fixed subfield $\bm{Q}(x_1,\ldots,x_n)^G:=\{f\in
\bm{Q}(x_1,\ldots,x_n): \sigma\cdot f=f$ for all $\sigma\in G\}$
is rational (i.e.\ purely transcendental) over $\bm{Q}$.

For any $\lambda\in GL_n(\bm{Q})$, any finite subgroup $G$ of
$GL_n(\bm{Q})$, if
$\bm{Q}(x_1,\ldots,x_n)^G=\bm{Q}(f_1,\ldots,f_N)$, then
$\bm{Q}(x_1,\ldots,x_n)^{\lambda\cdot G\cdot
\lambda^{-1}}=\bm{Q}(\lambda(f_1),\ldots,\lambda(f_N))$. Thus the
rationality of $\bm{Q}(x_1,\ldots,x_n)^G$ depends only on the
conjugacy class of $G$ within $GL_n(\bm{Q})$. It is easy to see
that $\bm{Q}(x_2,x_2)^G$ is rational over $\bm{Q}$ for any finite
subgroup $G$ of $GL_2(\bm{Q})$. By a result of Oura and Rikuna
\cite{OR}, $\bm{Q}(x_1,x_2,x_3)^G$ is rational over $\bm{Q}$ for
any finite subgroup $G$ of $GL_3(\bm{Q})$ (also see Theorem A2 in
the appendix of this paper). The goal of this paper is to study
the rationality of $\bm{Q}(x_1,x_2,x_3,x_4)^G$ where $G$ is a
finite subgroup of $GL_4(\bm{Q})$.

There are precisely 227 finite subgroups up to conjugation
contained in $GL_4(\bm{Q})$. A complete list of these subgroups
can be found in the book of Brown, B\"ulow, Neub\"user,
Wondratschek and Zassenhaus \cite[pages 80--260]{BBNWZ}. There
these 227 groups are classified into 33 crystal systems. Each
crystal system contains one or more $\bm{Q}$-classes; each
$\bm{Q}$-class is a conjugacy class of some finite subgroup of
$GL_4(\bm{Q})$. Since every finite subgroup of $GL_n(\bm{Q})$ can
be realized as a finite subgroup of $GL_n(\bm{Z})$, each
$\bm{Q}$-class in \cite{BBNWZ} contains one or more
$\bm{Z}$-classes; these $\bm{Z}$-classes are not conjugate within
$GL_n(\bm{Z})$, but they are conjugate within $GL_n(\bm{Q})$ and
represent this $\bm{Q}$-class. A set of generators of each
$\bm{Z}$-class is exhibited in \cite{BBNWZ}. The notation
$(4,26,1)$ used in \cite{KY} means the first $\bm{Q}$-class in the
26-th crystal system of $GL_4(\bm{Q})$ (see \cite[page
232]{BBNWZ}). Note that these $\bm{Z}$-classes can be found also
in the data base of GAP at the command
``GeneratorsOfGroup(MatGroupZClass(4,33,3,1))" for the first
$\bm{Z}$-class in the $\bm{Q}$-class $(4,33,3)$. We will use the
same notation as in \cite{KY}.

We recall two known results of this question.

\begin{theorem}[Kitayama \cite{Ki}] \label{t1.1}
For $n=4$ or $5$, if $H$ is a finite $2$-group of
$\bm{Q}(x_1,\ldots,x_n)$, then $\bm{Q}(x_1,\ldots,x_n)^H$ is
rational over $\bm{Q}$ if and only if $H$ is not isomorphic to
$C_8$, the cyclic group of order $8$.
\end{theorem}

\begin{theorem}[Kitayama and Yamasaki \cite{KY}] \label{t1.2}
Let $G$ be a finite subgroup of $GL_4(\bm{Q})$.
If $G$ doesn't belong to the $6$ conjugacy classes $(4,26,1)$, $(4,33,2)$, $(4,33,3)$, $(4,33,6)$,
$(4,33,7)$, $(4,33,11)$, then $\bm{Q}(x_1,x_2,x_3,x_4)^G$ is rational over $\bm{Q}$.
If $G$ is conjugate to $(4,26,1)$ or $(4,33,2)$,
then $\bm{Q}(x_1,x_2,x_3,x_4)^G$ is not rational over $\bm{Q}$.
\end{theorem}

The main result of this paper is to solve the four case $(4,33,3)$, $(4,33,6)$, $(4,33,7)$, $(4,33,11)$ left
unsettled in Theorem \ref{t1.2}.
Since the sets of generators of these four groups in \cite[pages 377--378]{KY} are defined over $\bm{Z}[1/2]$
(see Section 3),
we may consider similar rationality problems for a field $k$ with $\fn{char}k\ne 2$.
Here is our result.

\begin{theorem} \label{t1.3}
{\rm (1)} Let $k$ be a field with $\fn{char}k\ne 2$ and $G$ be a
finite group belong to the conjugacy class $(4,33,3)$ or
$(4,33,6)$, which is defined on \cite[page 377]{KY}. Then both
$k(x_1,x_2,x_3,x_4)^G$ and $k(x_1/x_4,x_2/x_4,x_3/x_4)^G$ are
rational over $k$.

{\rm (2)} Let $k$ be a field with $\fn{char}k\ne 2,3$ and $G$ be a
finite group belong to the conjugacy class $(4,33,7)$ or
$(4,33,11)$ which is defined on \cite[page 378]{KY}. Then both
$k(x_1,x_2,x_3,x_4)^G$ and $k(x_1/x_4,x_2/x_4,x_3/x_4)^G$ are
rational over $k$.
\end{theorem}

Combining Theorem \ref{t1.2} and Theorem \ref{t1.3}, we obtain the following result.

\begin{theorem} \label{t1.4}
Let $G$ be a finite subgroup of $GL_4(\bm{Q})$.
Then $\bm{Q}(x_1,x_2,x_3,x_4)^G$ is rational over $\bm{Q}$ if and only if $G$ is not
conjugate to the $(4,26,1)$ or $(4,33,2)$.
\end{theorem}

Note that the groups $(4,26,1)$ and $(4,33,2)$ are images of
faithful representations of $C_8$ and $C_3\rtimes C_8$ into
$GL_4(\bm{Q})$ respectively. The fixed fields for these two groups
are not rational over $\bm{Q}$ by Voskresenskii, Lenstra and
Saltman by \cite[Theorem 5.1, Theorem 5.11, Theorem 3.1]{Sa},
since Theorem 5.1 of \cite{Sa} is valid for $F(V)^G$ where $G\to
GL(V)$ is any faithful representation of $G$.

The main idea of the proof of Theorem \ref{t1.3} is to enlarge the
field $k$ to $K$ by adding $\sqrt{-1}$, $\sqrt{2}$ or $\sqrt{-3}$
to $k$ so that the given representations of these groups will
become more simple. Then consider the fixed subfield of ``the
projective part", i.e.\ $K(y_1/y_4,y_2/y_4,y_3/y_4)^G$, with the
aid of Theorem \ref{t2.1} in Section 2. Since we have enlarged $k$
to $K$, we should descend the ground field from $K$ to $k$, which
leads to a field of the form of two successive conic bundles. The
method to show the rationality of the resulting field is
elementary, but technical. Moreover, in the proof of Theorem
\ref{t1.3}, lots of computations are necessary. Some of them can
be carried out by hands as most proof of the ``traditional"
mathematical theorems. But we use the computer algebra package in
some other computations for the sake of saving energy and keeping
free from possible ``manual" mistakes. We emphasize that our use
of the computer in this paper is limited only to the routine
symbolic computation; no extra codes of data bases, e.g. GAP, are
required.

We will note that the rationality of $k(x_1,x_2,x_3,x_4)^G$
depends on the group $G$ as an abstract group and also depends on
the faithful representation of $G$ into $GL_4(\bm{Q})$. For
example, both the groups $(4,32,11)$ and $(4,33,6)$ are isomorphic
to $GL_2(\bm{F}_3)$ as abstract groups (see \cite[page 377]{KY}).
The fixed field $k(x_1,x_2,x_3,x_4)^G$ is rational for the group
$(4,32,11)$ by \cite{Pl,Ri}. Let $k(x'_1,x'_2,x'_3,x'_4)$ be the
field with the action of $G$ associated to $(4,33,6)$. By applying
Theorem \ref{t2.2} in Section 2 to
$k(x_1,x_2,x_3,x_4)(x'_1,x'_2,x'_3,x'_4)$ and
$k(x'_1,x'_2,x'_3,x'_4)(x_1,x_2,x_3,x_4)$, we find easily that
$k(x'_1,x'_2,x'_3,x'_4)^G(t_1,t_2,t_3,t_4)$ is rational over $k$
(where $g(t_i)=t_i$ for any $g\in G$, any $1\le i\le 4$). Thus
$k(x'_1,x'_2,x'_3,x'_4)^G$ is stably rational over $k$. But it is
not obvious at all whether $k(x'_1,x'_2,x'_3,x'_4)^G$ is rational
over $k$. The same situation holds for the groups $(4,32,5)$ and
$(4,33,3)$.

We will organize this paper as follows. We recall some
preliminaries in Section 2, which will be used in the proof of
Theorem \ref{t1.3}. In Section 3, we give the sets of generators
of the four groups $(4,33,3)$, $(4,33,6)$, $(4,33,7)$ and
$(4,33,11)$. The rationality for the groups $(4,33,3)$ and
$(4,33,6)$ will be given in Section 4 and Section 5 respectively.
The rationality of the groups $(4,33,7)$ and $(4,33,11)$ is given
in Section 6. Thus the proof of Theorem 1.3 is finished. Since the
rationality of $(4,31,i)$ where $3\le i\le 7$ in \cite[page 368,
lines 6--8 from the bottom]{KY} was referred to an unpublished
preprint of Yamasaki \cite{Ya1}, we include a proof of these cases
in the appendix for the convenience of the reader. For the same
reason we include a proof of Oura-Rikuna's Theorem, i.e. the
rationality for finite subgroups in $GL_3(\bm{Q})$ in the
appendix, because Oura and Rikuna's paper \cite{OR} hasn't been
published when this paper was written (although their preprint
appeared already in 2003). Since we don't have the preprints
\cite{OR} and \cite{Ya1}, we are not sure whether the proof in
Theorem A1 and Theorem A2 are the same as theirs.

Notation and terminology. The cyclic group of order $n$ will be
denoted by $C_n$. If $G$ is a finite subgroup of $GL_n(\bm{Z})$,
for any field $k$, we will say that $G$ acts on the rational
function field $k(x_1,\ldots,x_n)$ by monomial $k$-automorphisms,
if for any $\sigma=(a_{ij})_{1\le i,j\le n} \in G\subset
GL_n(\bm{Z})$, we define $\sigma\cdot
x_j=b_j(\sigma)\cdot\prod_{1\le i\le n} x_i^{a_{ij}}$ for some
$b_j(\sigma) \in k\backslash \{0\}$.

\section{Preliminaries}

We recall several results which will be used in tackling the rationality problem.

\begin{theorem}[Ahmad, Hajja and Kang {\cite[Theorem 3.1]{AHK}}] \label{t2.1}
Let $L$ be any field,
$L(x)$ the rational function field of one variable over $L$ and $G$ a finite group acting on $L(x)$.
Suppose that, for any $\sigma \in G$, $\sigma(L)\subset L$ and $\sigma(x)=a_\sigma\cdot x+b_\sigma$
where $a_\sigma,b_\sigma\in L$ and $a_\sigma\ne 0$.
Then $L(x)^G=L^G(f)$ for some polynomial $f\in L[x]$.
In fact, if $m=\min\{\deg g(x)\ge 1: g(x)\in L[x]^G\}$,
any polynomial $f\in L[x]^G$ with $\deg f=m$ satisfies the property $L(x)^G=L^G(f)$.
\end{theorem}

\begin{theorem}[Hajja and Kang {\cite[Theorem 1]{HK}}] \label{t2.2}
Let $G$ be a finite group acting on $L(x_1,\ldots,x_n)$,
the rational function field of $n$ variables over a field $L$.
Suppose that
\begin{enumerate}
\item[{\rm (i)}] for any $\sigma\in G$, $\sigma(L)\subset L$;
\item[{\rm (ii)}] the restriction of the action of $G$ to $L$ is faithful;
\item[{\rm (iii)}] for any $\sigma\in G$,
\[
\begin{pmatrix} \sigma(x_1) \\ \sigma(x_2) \\ \vdots \\ \sigma(x_n) \end{pmatrix}
=A(\sigma)\cdot \begin{pmatrix} x_1 \\ x_2 \\ \vdots \\ x_n \end{pmatrix}+B(\sigma)
\]
where $A(\sigma)\in GL_n(L)$ and $B(\sigma)$ is a $n\times 1$ matrix over $L$.
\end{enumerate}
\end{theorem}

Then there exist elements $z_1,\ldots,z_n\in L(x_1,\ldots,x_n)$ which are algebraically independent over $L$,
and $L(x_1,\ldots,x_n)=L(z_1,\ldots,z_n)$ so that $\sigma(z_i)=z_i$ for any $\sigma \in G$,
any $1\le i\le n$.

\begin{theorem}[Kang {\cite[Theorem 2.4]{Ka}}] \label{t2.3}
Let $k$ be any field, $\sigma$ be a $k$-automorphism of the rational function field $k(x,y)$ defined by
$\sigma(x)=a/x$, $\sigma(y)=b/y$ where $a\in k\backslash\{0\}$,
$b=c[x+(a/x)]+d$ such that $c,d \in k$ and at least one of $c$ and $d$ is non-zero.
Then $k(x,y)^{\langle\sigma\rangle}=k(u,v)$ where $u$ and $v$ are defined as
\[
u=\frac{x-(a/x)}{xy-(ab/xy)}\,,~~ v = \frac{y-(b/y)}{xy-(ab/xy)}\,.
\]
\end{theorem}

\begin{theorem}[Yamasaki \cite{Ya2}] \label{t2.4}
Let $k$ be a field with $\fn{char}k\ne 2$, $a\in k\backslash\{0\}$,
$\sigma$ be a $k$-automorphism of the rational function field $k(x,y)$ defined by $\sigma(x)=a/x$, $\sigma(y)=a/y$.
Then $k(x,y)^{\langle \sigma\rangle}=k(u,v)$ where $u=(x-y)/(a-xy)$, $v=(x+y)/(a+xy)$.
\end{theorem}

\begin{theorem}[Masuda {\cite[Theorem 3; HoK2, Theorem 2.2]{Ma}}] \label{t2.5}
Let $k$ be any field, $\sigma$ be a $k$-automorphism of the
rational function field $k(x,y,z)$ defined by $\sigma: x\mapsto
y\mapsto z\mapsto x$. Then
$k(x,y,z)^{\langle\sigma\rangle}=k(s_1,u,v)=k(s_3,u,v)$ where
$s_1$, $s_2$, $s_3$ are the elementary symmetric functions of
degree one, two, three in $x$, $y$, $z$, and $u$ and $v$ are
defined as
\begin{align*}
u &= \frac{x^2y+y^2z+z^2x-3xyz}{x^2+y^2+z^2-xy-yz-zx}, \\
v &= \frac{xy^2+yz^2+zx^2-3xyz}{x^2+y^2+z^2-xy-yz-zx}.
\end{align*}
\end{theorem}

\begin{remark}
The formula of $u$ and $v$ was essentially due to Masuda \cite[Theorem 3]{Ma} with a misprint
in the original expression.
The error was corrected by Rikuna \cite{Ri}.
For the details,
see the paragraph before Theorem 2.2 of \cite{HoK1}.
\end{remark}

\begin{theorem}[Hajja \cite{Ha}] \label{t2.6}
Let $k$ be any field, and $G$ be a finite group acting on the rational function field $k(x,y)$ by
monomial $k$-automorphisms.
Then $k(x,y)^G$ is rational over $k$.
\end{theorem}

\section{Sets of generators of the four groups}

In this section we recall the sets of generators of the groups $(4,33,3)$, $(4,33,6)$, $(4,33,7)$
and $(4,33,11)$ given in \cite[pages 377--378]{KY}.

Note that various matrices in $GL_4(\bm{Q})$ are defined on \cite[page 362]{KY} and
groups of crystal systems 33 are defined on \cite[page 369]{KY}.
In particular, the group $Q$ in the notation of \cite{KY} is isomorphic to the quaternion group of order 8,
and $Q$ is generated by $i$ and $ij$.
We write explicitly these matrices in $GL_4(\bm{Q})$ as follows:
\begin{align*}
i &= \left(\hs{-1}\begin{array}{cc@{\vds}cc}
0 & 1 & & \\ -1 & 0 & & \\[-3\jot] \hds{4} \\[-1\jot] & & 0 & 1 \\ & & -1 & 0
\end{array}\hs{-1}\right),  &
ij &= \left(\hs{-1}\begin{array}{cc@{\vds}cc}
& & 0 & -1 \\ & & -1 & 0 \\[-3\jot] \hds{4} \\[-1\jot] 0 & 1 & & \\ 1 & 0 & &
\end{array}\hs{-1}\right),    \\
\alpha &= \frac{1}{2}\left(\begin{array}{@{}rrrr@{}}
-1 & -1 & -1 & 1 \\ 1 & -1 & 1 & 1 \\ 1 & -1 & -1 & -1 \\ -1 & -1 & 1 & -1
\end{array}\right), &
\alpha_0 &= \left(\hs{-1}\begin{array}{c@{}c@{}cc@{}c@{}c}
-1 & \vds & & & & \\[-3\jot] \hds{5} & \\[-1\jot]
& \vds & 0 & -1 & \vds & \\
& \vds & -1 & 0 & \vds & \\[-3\jot] & \hds{5} \\[-1\jot]
& & & & \vds & 1
\end{array}\hs{-1}\right), \\
\alpha_1 &=\frac{1}{2}\left(\begin{array}{@{}rrrr@{}}
-1 & -1 & -1 & -1 \\ 1 & -1 & 1 & -1 \\ 1 & -1 & -1 & 1 \\ 1 & 1 & -1 & -1
\end{array}\right), &
k_1\alpha_0 &= \left(\hs{-1}\begin{array}{c@{}c@{}cc@{}c@{}c}
& & & & \vds & 1 \\[-3\jot] & \hds{5} \\[-1\jot]
& \vds & 1 & 0 & \vds & \\
& \vds & 0 & -1 & \vds & \\[-3\jot] \hds{5} & \\[-1\jot]
1 & \vds & & & &
\end{array}\hs{-1}\right).
\end{align*}

We will define a matrix $T\in GL_4(\bm{Q})$ by
\[
T=\left(\hs{-1}\begin{array}{cc@{\vds}cc}
0 & 1 & & \\ 1 & 0 & & \\[-3\jot] \hds{4} \\[-1\jot]
& & 1 & 0 \\ & & 0 & 1
\end{array}\hs{-1}\right).
\]

We define $\lambda_1$, $\lambda_2$, $\sigma$, $\tau$, $\lambda_3$, $\lambda_4$ which are conjugates of
$i$, $ij$, $\alpha$, $\alpha_0$, $\alpha_1$, $k_1\alpha_0$ by the conjugation associated to $T$.
Explicitly, we get
\begin{align*}
\lambda_1 &= T\cdot i\cdot T^{-1}=\left(\hs{-1}\begin{array}{cc@{\vds}cc}
0 & -1 & & \\ 1 & 0 & & \\[-3\jot] \hds{4} \\[-1\jot] & & 0 & 1 \\ & & -1 & 0
\end{array}\hs{-1}\right), &
\lambda_2 &= T\cdot ij \cdot T^{-1}=\left(\hs{-1}\begin{array}{cc@{\vds}cc}
& & -1 & 0 \\ & & 0 & -1 \\[-3\jot] \hds{4} \\[-1\jot] 1 & 0 & & \\ 0 & 1 & &
\end{array}\hs{-1}\right), \\
\sigma &= T\cdot \alpha\cdot T^{-1}=\frac{1}{2}\left(\begin{array}{@{}rrrr@{}}
-1 & 1 & 1 & 1 \\ -1 & -1 & -1 & 1 \\ -1 & 1 & -1 & -1 \\ -1 & -1 & 1 & -1
\end{array}\right), &
\tau &= T\cdot\alpha_0\cdot T^{-1}=\left(\hs{-1}\begin{array}{ccc@{\vds}c}
& & -1 & \\ & -1 & & \\ -1 & & & \\[-3\jot] \hds{4} \\[-1\jot] & & & 1
\end{array}\hs{-1}\right), \\
\lambda_3 &= T\cdot\alpha_1\cdot T^{-1}=\frac{1}{2}\left(\begin{array}{@{}rrrr@{}}
-1 & 1 & 1 & -1 \\ -1 & -1 & -1 & -1 \\ -1 & 1 & -1 & 1 \\ 1 & 1 & -1 & -1
\end{array}\right), &
\lambda_4 &= T\cdot k_1\alpha_0\cdot T^{-1}=\left(\hs{-1}\begin{array}{c@{\vds}ccc}
1 & & & \\[-3\jot] \hds{4} \\[-1\jot] & & & 1 \\ & & -1 & \\ & 1 & &
\end{array}\hs{-1}\right).
\end{align*}

\begin{defn} \label{d3.1}
The group $(4,33,3)$ is generated by $\lambda_1$, $\lambda_2$, $\sigma$;
the group $(4,33,6)$ is generated by $\lambda_1$, $\lambda_2$, $\sigma$, $\tau$;
the group $(4,33,7)$ is generated by $\lambda_1$, $\lambda_2$, $\sigma$, $\lambda_3$;
the group $(4,33,11)$ is generated by $\lambda_1$, $\lambda_2$, $\sigma$, $\lambda_3$, $\lambda_4$.
\end{defn}

Note that, as abstract groups, the group $(4,33,3)$ is isomorphic
to the group $SL_2(\bm{F}_3)$, the group $(4,33,6)$ is isomorphic
to the group $GL_2(\bm{F}_3)$, the group $(4,33,7)$ is a central
extension of the group $SL_2(\bm{F}_3)$, while the group
$(4,33,11)$ is of order $144$.

\section{The rationality of the group (4,33,3)}

Throughout this section $k$ is any field with $\fn{char}k\ne 2$,
and $G$ is the group $(4,33,3)$.
Thus $G=\langle \lambda_1,\lambda_2,\sigma\rangle$ by Definition \ref{d3.1}.
$G$ acts on $k(x_1,x_2,x_3,x_4)$ by $k$-automorphisms defined as
\begin{equation}
\begin{aligned}
\lambda_1:{} & x_1 \mapsto x_2,~ x_2\mapsto -x_1,~ x_3\mapsto -x_4,~ x_4\mapsto x_3, \\
\lambda_2:{} & x_1\mapsto x_3\mapsto -x_1,~ x_2\mapsto x_4\mapsto -x_2, \\
\sigma:{} & x_1\mapsto -(x_1+x_2+x_3+x_4)/2,~ x_2\mapsto (x_1-x_2+x_3-x_4)/2, \\
& x_3\mapsto (x_1-x_2-x_3+x_4)/2,~ x_4\mapsto (x_1+x_2-x_3-x_4)/2.
\end{aligned} \label{4-1}
\end{equation}

Step 1.
Let $\pi=\fn{Gal}(k(\sqrt{-1})/k)$.
Note that, if $\sqrt{-1}\in k$, then $\pi=\{1\}$; if $\sqrt{-1} \notin k$,
then $\pi=\langle \rho\rangle$ with $\rho(\sqrt{-1})=-\sqrt{-1}$.

We can extend the actions of $G$ and $\pi$ to $k(\sqrt{-1})(x_1,x_2,x_3,x_4)$ by requiring that
$G$ acts trivially on $k(\sqrt{-1})$ and $\pi$ acts trivially on $x_1$, $x_2$, $x_3$, $x_4$.
It follows that $k(x_1,x_2,x_3,x_4)^G=\{k(\sqrt{-1})(x_1,x_2,x_3,x_4)^{\pi}\}^G
=k(\sqrt{-1})(x_1,x_2,x_3,x_4)^{\langle G,\pi\rangle}$.

\bigskip

Step 2.
Define $y_1$, $y_2$, $y_3$, $y_4$ by
\begin{equation}
y_1=x_1-\sqrt{-1}x_2,~ y_2=x_3-\sqrt{-1}x_4,~ y_3=x_3+\sqrt{-1}x_4,~ y_4=-x_1-\sqrt{-1}x_2. \label{4-2}
\end{equation}
In other words, we define $y_1$, $y_2$, $y_3$, $y_4$ by
\[
\begin{pmatrix} y_1 \\ y_2 \\ y_3 \\ y_4  \end{pmatrix} =
\begin{pmatrix} 1 & -\sqrt{-1} & 0 & 0 \\ 0 & 0 & 1 & -\sqrt{-1} \\
0 & 0 & 1 & \sqrt{-1} \\ -1 & -\sqrt{-1} & 0 & 0 \end{pmatrix}
\begin{pmatrix} x_1 \\ x_2 \\ x_3 \\ x_4  \end{pmatrix}.
\]

It follows that $k(\sqrt{-1})(x_1,x_2,x_3,x_4)=k(\sqrt{-1})(y_1,y_2,y_3,y_4)$ and the actions of $G$
and $\pi$ (if $\sqrt{-1}\notin k$) are given by
\begin{equation}
\begin{aligned}
\lambda_1:{} & y_1\mapsto \sqrt{-1} y_1,~ y_2\mapsto -\sqrt{-1} y_2,~
  y_3\mapsto \sqrt{-1} y_3,~ y_4\mapsto -\sqrt{-1} y_4,  \\
\lambda_2:{} & y_1\mapsto y_2\mapsto -y_1,~ y_3\mapsto y_4\mapsto -y_3, \\
\sigma:{} & y_1\mapsto (-1-\sqrt{-1})(y_1+y_2)/2,~ y_2\mapsto (1-\sqrt{-1})(y_1-y_2)/2, \\
& y_3\mapsto (-1-\sqrt{-1})(y_3+y_4)/2,~ y_4\mapsto (1-\sqrt{-1})(y_3-y_4)/2, \\
\rho:{} & \sqrt{-1}\mapsto -\sqrt{-1},~ y_1\mapsto -y_4,~ y_2\leftrightarrow y_3,~ y_4\mapsto -y_1.
\end{aligned} \label{4-3}
\end{equation}

\bigskip

Step 3. By the same arguments as in Step 1, it is easy to see that
$k(x_1/x_4,x_2/x_4,$ $x_3/x_4)^G =
k(\sqrt{-1})(y_1/y_2,y_3/y_4,y_1/y_3)^{\langle G,\pi\rangle}$.

Define $z_1=y_1/y_2$, $z_2=y_3/y_4$, $z_3=y_1/y_3$. By Theorem
\ref{t2.1}, we find that $k(\sqrt{-1})(y_1,y_2,y_3,y_4)^{\langle
G,\pi\rangle} =k(\sqrt{-1})(z_1,z_2,z_3)(y_4)^{\langle
G,\pi\rangle}=k(\sqrt{-1})(z_1,z_2,z_3)^{\langle
G,\pi\rangle}(z_0)$ where $z_0$ is fixed by the actions of $G$ and
$\pi$.

It remains to show that $k(\sqrt{-1})(z_1,z_2,z_3)^{\langle
G,\pi\rangle}$ is rational over $k$.

\bigskip

Step 4.
Define $u_1=z_1/z_2$, $u_2=z_1z_2$, $u_3=z_3$.
Then $k(\sqrt{-1})(z_1,z_2,z_3)^{\langle\lambda_1\rangle}=k(\sqrt{-1})(u_1,u_2,u_3)$.
Moreover, $\lambda_2$ acts on $u_1$, $u_2$, $u_3$ by
\[
\lambda_2: u_1 \mapsto 1/u_1,~ u_2\mapsto 1/u_2,~ u_3\mapsto u_3/u_1.
\]

Define $v_3=u_3(1+(1/u_1))$.
Then $k(\sqrt{-1})(u_1,u_2,u_3)=k(\sqrt{-1})(u_1,u_2,v_3)$.
By Theorem \ref{t2.4},
we find that $k(\sqrt{-1})(u_1,u_2,v_3)^{\langle\lambda_2\rangle}=k(\sqrt{-1})(v_1,v_2,v_3)$
where $v_1= (u_1-u_2)/(1-u_1u_2)$, $v_2=(u_1+u_2)/(1+u_1u_2)$.
In summary, we get $k(\sqrt{-1})(z_1,z_2,\break z_3)^{\langle \lambda_1,\lambda_2\rangle}=k(\sqrt{-1})(v_1,v_2,v_3)$.

\bigskip

Step 5.
Note that the actions of $\sigma$ and $\rho$ (if $\sqrt{-1}\notin k$) are given by
\begin{align*}
\sigma:{} & v_1\mapsto v_2/v_1,~ v_2\mapsto 1/v_1,~ v_3\mapsto v_3(v_1+v_2)/(v_1(1+v_2)) \\
\rho:{} & \sqrt{-1}\mapsto -\sqrt{-1},~ v_1\mapsto 1/v_1,~ v_2\mapsto 1/v_2,~
v_3\mapsto -2(1+v_1)(1+v_2)/(v_3(v_1+v_2)).
\end{align*}

Define $X_3=v_3(1+v_1+v_2)/((1+v_1)(1+v_2))$.
Then $k(\sqrt{-1})(v_1,v_2,v_3)=k(\sqrt{-1})\break (v_1,v_2,X_3)$ and $\sigma(X_3)=X_3$.

Thus $k(\sqrt{-1})(v_1,v_2,v_3)^{\langle\sigma\rangle}=k(\sqrt{-1})(v_1,v_2)^{\langle\sigma\rangle}(X_3)
=k(\sqrt{-1})(X_1,X_2,X_3)$ by Theorem \ref{t2.5} (regarding $v_1$, $v_2/v_1$, $1/v_2$ as $x$, $y$, $z$ in
Theorem \ref{t2.5} and defining $X_1$ and $X_2$ as $u$ and $v$ there) where $X_1$ and $X_2$ are defined by
\begin{align*}
X_1 &= (v_1^3v_2^3+v_1^3+v_2^3-3v_1^2v_2^2)/(v_1^4v_2^2+v_2^4+v_1^2-v_1^2v_2^3-v_1v_2^2-v_1^3v_2), \\
X_2 &=
(v_1v_2^4+v_1v_2+v_1^4v_2-3v_1^2v_2^2)/(v_1^4v_2^2+v_2^4+v_1^2-v_1^2v_2^3-v_1v_2^2-v_1^3v_2).
\end{align*}

In conclusion, $k(\sqrt{-1})(z_1,z_2,z_3)^G=k(\sqrt{-1})(X_1,X_2,X_3)$.

\bigskip

Step 6. If $\sqrt{-1}\in k$, we find that $\pi=\{1\}$. Hence
$k(\sqrt{-1})(z_1,z_2,z_3)^G=k(z_1,z_2,z_3)^G$ $=k(X_1,X_2,X_3)$
is rational over $k$.

From now on, we assume that $\sqrt{-1}\notin k$ and $\pi=\langle
\rho\rangle$. We will show that $k(\sqrt{-1})\break
(X_1,X_2,X_3)^{\langle\rho\rangle}$ is rational over $k$.

We use the computer to perform the action of $\rho$ on $X_1$, $X_2$, $X_3$.
We get
\[
\rho: X_1\mapsto X_2/(X_1^2-X_1X_2+X_2^2), ~X_2\mapsto X_1/(X_1^2-X_1X_2+X_2^2),~ X_3\mapsto -2A/X_3
\]
where $A=g_1g_2g_3^{-1}$ and
\begin{align*}
g_1 &= (1+X_1)^2-X_2(1+X_1)+X_2^2, ~~ g_2= (1+X_2)^2-X_1(1+X_2)+X_1^2, \\
g_3 &= 1+X_1+X_2+X_1^3+X_2^3+X_1X_2(3X_1X_2-2X_1^2-2X_2^2+2)+X_1^4+X_2^4.
\end{align*}

Note that $\rho(g_1)=g_2/(X_1^2-X_1X_2+X_2^2)$.

Define $Y_1=X_1/X_2$, $Y_2=X_1$, $Y_3=X_3/g_1$.
Then $k(\sqrt{-1})(X_1,X_2,X_3)=k(\sqrt{-1})\break (Y_1,Y_2,Y_3)$ and
\[
\rho: Y_1\mapsto 1/Y_1,~ Y_2\mapsto Y_1/(Y_2(1-Y_1+Y_1^2)),~ Y_3\mapsto B/Y_3
\]
where $B=-2Y_1^2Y_2^2(1-Y_1+Y_1^2)/[Y_1^4+Y_2^4+Y_1Y_2(Y_1^2+Y_2^2)+Y_1Y_2(Y_1^3-2Y_2^3+2Y_1^2Y_2)
+3Y_1^2Y_2^4+Y_1^3Y_2^3(Y_1-2Y_2)+Y_1^4Y_2^4]$.

Note that $k(\sqrt{-1})(Y_1,Y_2)^{\langle\rho\rangle}$ is the
function field of a conic bundle over
$k(\sqrt{-1})(Y_1)^{\langle\rho\rangle}$ (which is the function
field of $\bm{P}^1_k$), and
$k(\sqrt{-1})(Y_1,Y_2,Y_3)^{\langle\rho\rangle}$ is the function
field of another conic bundle over some $k$-surface whose function
field is $k(\sqrt{-1})(Y_1,Y_2)^{\langle\rho\rangle}$.

\bigskip

Step 7.
It is not difficult to verify that $k(\sqrt{-1})(Y_1,Y_2,Y_3)^{\langle\rho\rangle}=k(U_0,U_1,U_2,U_3,U_4)$
where $U_i$'s are defined by
\begin{align*}
U_0 &= \sqrt{-1}(1-Y_1)/(1+Y_1), ~~ U_1=Y_2+Y_1/(Y_2(1-Y_1+Y_1^2)), \\
U_2 &= \sqrt{-1}[Y_2-Y_1/(Y_2(1-Y_1+Y_1^2))], ~~ U_3=Y_3+B/Y_3, ~~ U_4=\sqrt{-1}(Y_3-B/Y_3)
\end{align*}
with the relations
\begin{equation}
U_1^2+U_2^2=4(1+U_0^2)/(1-3U_0^2) \mbox{ ~ and ~ } U_3^2+U_4^2=4B. \label{4-4}
\end{equation}

We will simplify the two relations in Formula \eqref{4-4}.

For example, in the first relation of Formula \eqref{4-4},
multiply both sides by $(1+U_0^2)$. Use the identity
$(a^2+b^2)(c^2+d^2)=(ac-bd)^2+(ad+bc)^2$ to simplify the
left-hand-side of the resulting relation. In other words, define
\[
V_1=(U_0U_1+U_2)(1-3U_0^2)/(2+2U_0^2), ~~ V_2=(U_0U_2-U_1)(1-3U_0^2)/(2+2U_0^2).
\]

We find the $k(U_0,U_1,U_2)=k(U_0,V_1,V_2)$ and the relation becomes
\[
V_1^2+V_2^2=1-3U_0^2.
\]

The above relation can be written as
\begin{equation}
[V_1^2/(1+V_2)^2]-[(1-V_2)/(1+V_2)]=-3U_0^2/(1+V_2)^2. \label{4-5}
\end{equation}

Define $W_1=V_1/(1+V_2)$, $W_2=U_0/(1+V_2)$.
The relation \eqref{4-5} guarantees that $(1-V_2)/(1+V_2) \in k(W_1,W_2)$.
Thus $V_2\in k(W_1,W_2)$.
It follows that $k(U_0,V_1,V_2)=k(W_1,W_2)$.

Now we rewrite the second relation $U_3^2+U_4^2=4B$ of Formula \eqref{4-4} in terms of $W_1$, $W_2$.
We get
\begin{equation}
U_3^2+U_4^2=[4W_1^2+(W_1^2+3W_2^2-1)^2]/[W_1^2-W_2^2-(W_1^2+3W_2^2)^2]. \label{4-6}
\end{equation}

Use the similar trick as above to simplify the relation \eqref{4-6}.
In short, define
\begin{align*}
W_3 &= [U_3(W_1^2{+}3W_2^2{-}1)+2U_4W_1][W_1^2-W_2^2-(W_1^2{+}3W_2^2)^2]/[4W_1^2+(W_1^2{+}3W_2^2{-}1)^2], \\
W_4 &= [U_4(W_1^2{+}3W_2^2{-}1)-2U_3W_1][W_1^2-W_2^2-(W_1^2{+}3W_2^2)^2]/[4W_1^2+(W_1^2{+}3W_2^2{-}1)^2].
\end{align*}

We get $k(W_1,W_2,U_3,U_4)=k(W_1,W_2,W_3,W_4)$ and the relation \eqref{4-6} becomes
\begin{equation}
W_3^2+W_4^2-W_1^2+W_2^2+(W_1^2+3W_2^2)^2=0. \label{4-7}
\end{equation}

Define $w_1=W_1/(W_1^2+3W_2^2)$, $w_2=W_2/(W_1^2+3W_2^2)$, $w_3=W_3/(W_1^2+3W_2^2)$, $w_4=W_4/(W_1^2+3W_2^2)$.
We find that $k(W_1,W_2,W_3,W_4)=k(w_1,w_2,w_3,w_4)$ and the relation \eqref{4-7} becomes
\[
w_3^2+w_4^2-w_1^2+w_2^2+1=0.
\]

The above relation can be written as
\[
w_3^2+w_4^2+1=(w_1-w_2)(w_1+w_2).
\]

Thus $w_1+w_2\in k(w_1-w_2,w_3,w_4)$.
Hence $w_1,w_2\in k(w_1-w_2,w_3,w_4)$.
It follows that $k(w_1,w_2,w_3,w_4)=k(w_1-w_2,w_3,w_4)$ is rational over $k$. \qed

\section{The rationality of the group (4,33,6)}

Throughout this section $k$ is any field with $\fn{char}k\ne 2$, and $G$ is the group $(4,33,6)$.
Thus $G=\langle \lambda_1,\lambda_2,\sigma,\tau \rangle$ by Definition \ref{d3.1}.
The actions of $\lambda_1$, $\lambda_2$, $\sigma$ on $k(x_1,x_2,x_3,x_4)$ is the same
as those given in Formula \eqref{4-1}.
We record the action of $\tau$ as follows
\[
\tau: x_1\mapsto -x_3,~ x_2\mapsto -x_2,~ x_3\mapsto -x_1,~ x_4\mapsto x_4.
\]

The method to prove that $k(x_1,x_2,x_3,x_4)^G$ is $k$-rational is very similar to the method used in Section 4.
In many situations, even the formulae of changing variables are identically the same.

Step 1. Let $\pi=\fn{Gal}(k(\sqrt{-1},\sqrt{2})/k)$. Note that
$\pi$ is isomorphic to $\{1\}$, $C_2$ or $C_2\times C_2$. In case
$[k(\sqrt{-1},\sqrt{2}):k]=4$, define $\rho_1$, $\rho_2$ and
$\rho_3$ by $\rho_1(\sqrt{-1})=-\sqrt{-1}$,
$\rho_1(\sqrt{2})=\sqrt{2}$, $\rho_2(\sqrt{-1})=\sqrt{-1}$,
$\rho_2(\sqrt{2})=-\sqrt{2}$, $\rho_3(\sqrt{-1})=-\sqrt{-1}$,
$\rho_3(\sqrt{2})=-\sqrt{2}$. We get $\pi=\langle
\rho_1,\rho_2\rangle$ and $\rho_3=\rho_1 \rho_2$ in this
situation. In general, $\pi=\{1\}, \langle\rho_1\rangle,
\langle\rho_2\rangle, \langle\rho_3\rangle$ or $\langle
\rho_1,\rho_2\rangle$. In the sequel, whenever we describe the
action of $\rho_1$, we mean that $\sqrt{-1} \notin k$ and $\rho_1
\in \pi$; similarly for $\rho_2$ and $\rho_3$.

As in Step 1 of Section 4,
we may extend the actions of $G$ and $\pi$ to $k(\sqrt{-1},\sqrt{2})(x_1,\break x_2,x_3,x_4)$.

The Formula \eqref{4-2} should be modified.
We define $y_1$, $y_2$, $y_3$, $y_4$ by
\[
\begin{pmatrix} y_1 \\ y_2 \\ y_3 \\ y_4 \end{pmatrix}=
\begin{pmatrix} 1 & 0 & -1+\sqrt{2} & 0 \\ 0 & 1 & 0 & -1+\sqrt{2} \\
1-\sqrt{2} & 0 & 1 & 0 \\ 0 & 1-\sqrt{2} & 0 & 1 \end{pmatrix}
\begin{pmatrix} 1 & -\sqrt{-1} & 0 & 0 \\ 0 & 0 & 1 & -\sqrt{-1} \\
0 & 0 & 1 & -\sqrt{-1} \\ -1 & -\sqrt{-1} & 0 & 0 \end{pmatrix}
\begin{pmatrix} x_1 \\ x_2 \\ x_3 \\ x_4 \end{pmatrix}.
\]

By the same arguments as in Step 1 of Section 4, we can show that
$k(x_1,x_2,x_3,x_4)^G$ $=k(\sqrt{-1},\sqrt{2})(y_1,
y_2,y_3,y_4)^{\langle G, \pi \rangle}$ and
$k(x_1/x_4,x_2/x_4,x_3/x_4)^G=k(\sqrt{-1},\sqrt{2})(y_1/
y_2,y_3/y_4,$ $y_1/y_3)^{\langle G, \pi \rangle}$.

Moreover, the actions of $\lambda_1$, $\lambda_2$, $\sigma$ on
$y_1$, $y_2$, $y_3$, $y_4$ are the same as those given in Formula
\eqref{4-3}. The action of $\rho_1$ on $y_1$, $y_2$, $y_3$, $y_4$
is the same as the action of $\rho$ on $y_1$, $y_2$, $y_3$, $y_4$
given in Formula \eqref{4-3}; but remember that
$\rho_1(\sqrt{-1})=-\sqrt{-1}$, $\rho_1(\sqrt{2})=\sqrt{2}$. We
record the actions of $\tau$ and $\rho_2$ as follows:
\begin{align*}
\tau:{} & y_1\mapsto -(y_1+y_2)/\sqrt{2},~ y_2\mapsto (-y_1+y_2)/\sqrt{2},~
  y_3\mapsto (y_3+y_4)/\sqrt{2},~ y_4\mapsto (y_3-y_4)/\sqrt{2}, \\
\rho_2:{} & y_1\mapsto (-1-\sqrt{2})y_3,~ y_2\mapsto (-1-\sqrt{2})y_4,~
  y_3\mapsto (1+\sqrt{2})y_1,~ y_4\mapsto (1+\sqrt{2})y_2.
\end{align*}

It remains to show that $k(\sqrt{-1},\sqrt{2})(y_1/ y_2,y_3/y_4,$
$y_1/y_3)^{\langle G, \pi \rangle}$ is rational over $k$.
\bigskip

Step 2. The substitution formulae for $z_1$, $z_2$, $z_3$, $u_1$,
$u_2$, $u_3$, $v_1$, $v_2$, $v_3$, $X_1$, $X_2$, $X_3$ are
completely the same as in Step 3 $\sim$ Step 5 of Section 4. Thus
we get $k(\sqrt{-1},\sqrt{2})(y_1/y_2,$
$y_3/y_4,y_1/y_3)^{\langle\lambda_1,\lambda_2,\sigma\rangle}
=k(\sqrt{-1},\sqrt{2})(X_1,X_2,X_3)$ and
$k(\sqrt{-1},\sqrt{2})(y_1,
y_2,y_3,y_4)^{\langle\lambda_1,\lambda_2,\sigma\rangle}$
$=k(\sqrt{-1},\sqrt{2})(X_1,X_2,X_3,z_0)$ where $z_0$ is fixed by
all elements of $G$ and $\pi$.

Define $Y_1=X_1/X_2$, $Y_2=X_1$, $Y_3=X_1X_3/g_1$
(remember $A=g_1g_2g_3^{-1}$ and $A$ is defined in Step 6 of Section 4).
Note that $k(\sqrt{-1},\sqrt{2})(X_1,X_2,X_3)=k(\sqrt{-1},\sqrt{2})\break (Y_1,Y_2,Y_3)$ and $\tau$
acts on $X_i$, $Y_j$ as follows:
\begin{align*}
\tau:{} & X_1\mapsto X_1/(X_1^2-X_1X_2+X_2^2), ~ X_2\mapsto X_2/(X_1^2-X_1X_2+X_2^2),~ X_3\mapsto -X_3, \\
& Y_1\mapsto Y_1,~ Y_2\mapsto Y_1^2/(Y_2(1-Y_1+Y_1^2)),~ Y_3\mapsto -Y_3.
\end{align*}

It is easy to verify that $k(\sqrt{-1},\sqrt{2})(Y_1,Y_2,Y_3)^{\langle\tau\rangle}
=k(\sqrt{-1},\sqrt{2})(Z_1,Z_2,Z_3)$ where $Z_1=Y_1$, $Z_2=Y_2+[Y_1^2/(Y_2(1-Y_1+Y_1^2))]$,
$Z_3=Y_3\{Y_2-[Y_1^2/(Y_2(1-Y_1+Y_1^2))]\}$.

We conclude that $k(\sqrt{-1},\sqrt{2})(Y_1,Y_2,Y_3)^{\langle\tau\rangle}$ is rational over $k(\sqrt{-1},\sqrt{2})$.

If $k(\sqrt{-1},\sqrt{2})=k$, i.e.\ $\pi=\{1\}$, then
$k(x_1,x_2,x_3,x_4)^G$ is $k$-rational. In the remaining part of
this section we will consider the situations when
$\pi=\langle\rho_1,\rho_2\rangle$, $\langle\rho_1\rangle$,
$\langle\rho_2\rangle$ or $\langle\rho_3\rangle$.

\bigskip

Step 3.
We consider the case $\pi=\langle\rho_1,\rho_2\rangle$ first.

Using the computer for symbolic computation, we find that
$\rho_1(Z_i)=\rho_2(Z_i)$, $\rho_3(Z_i)=Z_i$ for $1\le i\le 3$ and
\[
\rho_1: Z_1\mapsto 1/Z_1, ~ Z_2\mapsto Z_2/Z_1,~ Z_3\mapsto C/Z_3
\]
where $C=2Z_1^2(-4Z_1^2+Z_2^2-Z_1Z_2^2+Z_1^2Z_2^2)/[(1-Z_1+Z_1^2)(-2Z_1^2+Z_1Z_2+Z_2^2+4Z_1^3-2Z_1Z_2^2
-2Z_1^4+3Z_1^2Z_2^2+Z_1^4Z_2-2Z_1^3Z_2^2+Z_1^4Z_2^2)]$.

It follows that
$k(\sqrt{-1},\sqrt{2})(Z_1,Z_2,Z_3)^{\langle\rho_1,\rho_2\rangle}=
\{k(\sqrt{-1},\sqrt{2})(Z_1,Z_2,Z_3)^{\langle\rho_1\rho_2\rangle}\}^{\langle\rho_1\rangle}
=k(\sqrt{-2})(Z_1, Z_2,Z_3)^{\langle\rho_1\rangle}$. The action of
$\rho_1$ is given by
\[
\rho_1: \sqrt{-2} \mapsto -\sqrt{-2}, ~ Z_1\mapsto 1/Z_1,~ Z_2\mapsto Z_2/Z_1, ~ Z_3\mapsto C/Z_3.
\]

\bigskip

Step 4.
It is not difficult to verify that $k(\sqrt{-2})(Z_1,Z_2,Z_3)^{\langle\rho_1\rangle}=k(U_1,U_2,U_3,U_4)$
where $U_1$, $U_2$, $U_3$, $U_4$ are defined by
\begin{alignat*}{2}
U_1 &= Z_2(1+(1/Z_1)), &~~ U_2 &= \sqrt{-2} Z_2(1-(1/Z_1)), \\
U_3 &= Z_3+(C/Z_3), & U_4 &= \sqrt{-2} (Z_3-(C/Z_3))
\end{alignat*}
with the relation
\begin{equation}
\begin{aligned}
2U_3^2+U_4^2 &= 16(-32+2U_1^2-3U_2^2)(2U_1^2+U_2^2)^2/[(2U_1^2-3U_2^2) (16U_1^3\\
&\qquad +4U_1^4+64U_2^2-24U_1U_2^2-12U_1^2U_2^2+9U_2^4)].
\end{aligned} \label{5-1}
\end{equation}

We will simplify the relation \eqref{5-1}. Use the same technique
as Step 7 in Section 4. Define
\[
V_1=8U_1/(2U_1^2-3U_2^2),~~ V_2=4U_2/(2U_1^2-3U_2^2),~~ V_3=\alpha
U_3, ~~ V_4=\alpha U_4
\]
where
$\alpha=(16U_1^3+4U_1^4+64U_2^2-24U_1U_2^2-12U_1^2U_2^2+9U_2^4)/[4(2U_1^2+U_2^2)(2U_1^2-3U_2^2)]$.

It follows that $k(U_1,U_2)=k(V_1,V_2)$,
$k(U_1,U_2,U_3,U_4)=k(V_1,V_2,V_3,V_4)$ and the relation
\eqref{5-1} becomes
\begin{equation}
2V_3^2+V_4^2=(1-V_1^2+6V_2^2)(1+V_1+4V_2^2). \label{5-2}
\end{equation}

Define $w_1=1/(1+V_1)$, $w_2=V_2/(1+V_1)$, $w_3=V_3/(1+V_1)^2$, $w_4=V_4/(1+V_1)^2$.
Then $k(V_1,V_2,V_3,V_4)=k(w_1,w_2,w_3,w_4)$ and the relation \eqref{5-2} becomes
\[
2w_3^2+w_4^2=(w_1+4w_2^2)(2w_1-1+6w_2^2).
\]

Thus we get a relation
\begin{equation}
2(w_3/(w_1+4w_2^2))^2+(w_4/(w_1+4w_2^2))^2=(2w_1-1+6w_2^2)/(w_1+4w_2^2). \label{5-3}
\end{equation}

It follows that $(2w_1-1+6w_2^2)/(w_1+4w_2^2)\in k(w_2,w_3/(w_1+4w_2^2),w_4/(w_1+4w_2^2))$.
Hence $w_1$ belongs to this field.
We find that $k(w_1,w_2,w_3,w_4)=k(w_2,w_3/(w_1+4w_2^2),\break w_4/(w_1+4w_2^2))$ is rational over $k$.

\bigskip

Step 5. Consider the case $\pi=\langle\rho_3\rangle$ i.e.\
$\sqrt{-2}\in k$. Hence
$k(\sqrt{-1},\sqrt{2})(Z_1,Z_2,Z_3)^{\langle\rho_3\rangle}=k(Z_1,Z_2,Z_3)$
is $k$-rational. Done.

Consider the case $\pi=\langle\rho_1\rangle$, i.e.\ $\sqrt{2}\in
k$. Then $k(\sqrt{-1},\sqrt{2})(Z_1,Z_2,Z_3)^{\langle\pi
\rangle}=k(\sqrt{-2})\break (Z_1,Z_2,Z_3)^{\langle\rho_1\rangle}$.
The action of $\rho_1$ is the same as that in the last line of
Step 3. The proof of rationality of the present situation is
completely the same as in Step 4. Done.

The case $\pi=\langle\rho_2\rangle$, i.e.\ $\sqrt{-1}\in k$, can
be discussed in a similar way (note that the actions of $\rho_1$
and $\rho_2$ are the same on $Z_1, Z_2, Z_3$). The details of the
proof is omitted. \qed

\section{The rationality of the groups (4,33,7) and (4,33,11)}

In this section we assume that $k$ is any field with $\fn{char}k\ne
2,3$. Let $G$ be the group $(4,33,7)$ or $(4,33,11)$. We will show
that $k(x_1,x_2,x_3,x_4)^G$ is rational over $k$.

By Definition \ref{d3.1}, if $G$ is the group $(4,33,7)$,
then $G=\langle \lambda_1,\lambda_2,\sigma,\lambda_3 \rangle$;
if $G$ is the group $(4,33,11)$,
then $G=\langle \lambda_1,\lambda_2,\sigma,\lambda_3,\lambda_4 \rangle$.

As in Section 5, the proof in this section is very similar to that in Section 4.

Step 1. Let $\pi=\fn{Gal}(k(\sqrt{-1},\sqrt{3})/k)$. If
$[k(\sqrt{-1},\sqrt{3}):k]=4$, define $\rho_1$, $\rho_2$ and
$\rho_3$ by $\rho_1(\sqrt{-1})=-\sqrt{-1}$,
$\rho_1(\sqrt{3})=\sqrt{3}$, $\rho_2(\sqrt{-1})=\sqrt{-1}$,
$\rho_2(\sqrt{3})=-\sqrt{3}$, $\rho_3(\sqrt{-1})=-\sqrt{-1}$,
$\rho_3(\sqrt{3})=-\sqrt{3}$. Then
$\pi=\langle\rho_1,\rho_2\rangle$. In general, $\pi=\{1\}$,
$\langle\rho_1\rangle$, $\langle\rho_2\rangle$,
$\langle\rho_3\rangle$ or $\langle\rho_1,\rho_2\rangle$.
Throughout this section, whenever we describe the action of
$\rho_1$, we mean that $\sqrt{-1} \notin k$ and $\rho_1 \in \pi$;
similarly for $\rho_2$ and $\rho_3$.

As before, we may extend the actions of $G$ and $\pi$ to $k(\sqrt{-1},\sqrt{3})(x_1,x_2,x_3,x_4)$.

The actions of $\lambda_1$, $\lambda_2$, $\sigma$ on $x_1$, $x_2$, $x_3$, $x_4$ are the
same as given in Formula \eqref{4-1}.
The actions of $\lambda_3$ and $\lambda_4$ are given by
\begin{align*}
\lambda_3:{} & x_1\mapsto (-x_1-x_2-x_3+x_4)/2,~ x_2\mapsto (x_1-x_2+x_3+x_4)/2, \\
& x_3\mapsto (x_1-x_2-x_3-x_4)/2,~ x_4\mapsto (-x_1-x_2+x_3-x_4)/2, \\
\lambda_4:{} & x_1\mapsto x_1,~ x_2\leftrightarrow x_4,~ x_3\mapsto -x_3.
\end{align*}

Define $y_1$, $y_2$, $y_3$, $y_4$ by
\begin{align*}
\begin{pmatrix} y_1 \\ y_2 \\ y_3 \\ y_4 \end{pmatrix} &=
\begin{pmatrix}
 1 & 0 & \frac{(-1+\sqrt{-1})(1+\sqrt{3})}{2} & 0 \\
 0 & 1 & 0 & \frac{(-1+\sqrt{-1})(1+\sqrt{3})}{2} \\
 \frac{(2+\sqrt{-1}+\sqrt{3})}{2} & 0 & \frac{(-\sqrt{-1}+\sqrt{3})}{2} & 0 \\
 0 & \frac{(2+\sqrt{-1}+\sqrt{3})}{2} & 0 & \frac{(-\sqrt{-1}+\sqrt{3})}{2}
\end{pmatrix} \\
&\qquad \times\begin{pmatrix}
 1 & -\sqrt{-1} & 0 & 0 \\ 0 & 0 & 1 & -\sqrt{-1} \\
 0 & 0 & 1 & \sqrt{-1} \\ -1 & -\sqrt{-1} & 0 & 0
\end{pmatrix}
\begin{pmatrix} x_1 \\ x_2 \\ x_3 \\ x_4 \end{pmatrix}.
\end{align*}

By the same arguments as in Step 1 of Section 4, we can show that
$k(x_1,x_2,x_3,x_4)^G$ $=k(\sqrt{-1},\sqrt{3})(y_1,
y_2,y_3,y_4)^{\langle G, \pi \rangle}$ and
$k(x_1/x_4,x_2/x_4,x_3/x_4)^G=k(\sqrt{-1},\sqrt{3})(y_1/ y_2,
\break y_3/y_4,y_1/y_3)^{\langle G, \pi \rangle}$.

The actions of $\lambda_1$, $\lambda_2$, $\sigma$ on $y_1$, $y_2$,
$y_3$, $y_4$ are the same as those given in Formula \eqref{4-3}.
We describe the actions of $\lambda_3$, $\lambda_4$, $\rho_1$,
$\rho_2$ and $\rho_3$:
\begin{align*}
\lambda_3:{}& y_1 \mapsto \zeta y_1,~ y_2\mapsto \zeta y_2,~ y_3\mapsto \zeta^2 y_3,~
  y_4 \mapsto \zeta^2 y_4 \mbox{ ~ (where } \zeta=(-1+\sqrt{-3})/2), \\
\lambda_4:{} & y_1\mapsto (1-\sqrt{-1})(y_3+y_4)/2,~ y_2\mapsto (1-\sqrt{-1})(y_3-y_4)/2,   \\
& y_3\mapsto (1+\sqrt{-1})(y_1+y_2)/2,~ y_4\mapsto (1+\sqrt{-1})(y_1-y_2)/2, \\
\rho_1:{} & y_1\mapsto -(\sqrt{-1}+\sqrt{3})y_4/2,~ y_2\mapsto (\sqrt{-1}+\sqrt{3})y_3/2, \\
& y_3\mapsto (\sqrt{-1}+\sqrt{3})y_2/2,~ y_4\mapsto -(\sqrt{-1}+\sqrt{3})y_1/2, \\
\rho_2:{} & y_1\mapsto -\sqrt{-1}(2+\sqrt{-1}-\sqrt{3})y_3/2,~ y_2\mapsto -\sqrt{-1}(2+\sqrt{-1}-\sqrt{3})y_4/2, \\
& y_3\mapsto (2+\sqrt{-1}-\sqrt{3})y_1/2,~ y_4\mapsto
(2+\sqrt{-1}-\sqrt{3})y_2/2, \\
\rho_3:{} & y_1\mapsto (1+\sqrt{-1})(-1+\sqrt{3})y_2/2,~ y_2\mapsto (1+\sqrt{-1})(1-\sqrt{3})y_1/2, \\
& y_3\mapsto (1-\sqrt{-1})(1-\sqrt{3})y_4/2,~ y_4\mapsto
(-1+\sqrt{-1})(1-\sqrt{3})y_3/2.
\end{align*}

It remains to show that $k(\sqrt{-1},\sqrt{3})(y_1/ y_2,y_3/y_4,$
$y_1/y_3)^{\langle G, \pi \rangle}$ is rational over $k$.

\bigskip

Step 2. The substitution formulae for $z_1$, $z_2$, $z_3$, $u_1$,
$u_2$, $u_3$, $v_1$, $v_2$, $v_3$, $X_1$, $X_2$, $X_3$ are
completely the same as in Step 3 $\sim$ Step 5 of Section 4. Thus
we get $k(\sqrt{-1},\sqrt{3})(y_1/y_2, \break
y_3/y_4,y_1/y_3)^{\langle\lambda_1,\lambda_2,\sigma\rangle}
=k(\sqrt{-1},\sqrt{3})(X_1,X_2,X_3)$ and
$k(\sqrt{-1},\sqrt{3})(y_1,y_2,
y_3,y_4)^{\langle\lambda_1,\lambda_2,\sigma\rangle}
=k(\sqrt{-1},\sqrt{3})(X_1,X_2,X_3,z_0)$ where $z_0$ is fixed by
all elements of $G$ and $\pi$.

The action of $\lambda_3$ is given by
\[
\lambda_3: X_1\mapsto X_1,~ X_2\mapsto X_2,~ X_3\mapsto \zeta^2 X_3.
\]

Hence $k(\sqrt{-1},\sqrt{3})(X_1,X_2,X_3)^{\langle\lambda_3\rangle}=k(\sqrt{-1},\sqrt{3})(X_1,X_2,X_3^3)$.

We will discuss the rationality problem for the group $(4,33,7)$ in the next two steps.
The discussion for the group $(4,33,11)$ will be postponed till Step 5.

\bigskip

Step 3.
Consider the group $G=\langle \lambda_1,\lambda_2,\sigma,\lambda_3\rangle$ in this step.

We have shown that
$k(\sqrt{-1},\sqrt{3})(z_1,z_2,z_3)^{\langle\lambda_1,\lambda_2,\sigma,\lambda_3\rangle}
=k(\sqrt{-1},\sqrt{3})(X_1,X_2,X_3^3)$. If $\pi=\{1\}$, then
$k(x_1,x_2,x_3,x_4)^G$ is $k$-rational. It remains to consider
cases when $\pi=\langle\rho_1,\rho_2\rangle$,
$\langle\rho_1\rangle$, $\langle\rho_2\rangle$,
$\langle\rho_3\rangle$. We consider the case
$\pi=\langle\rho_1,\rho_2\rangle$ first.

The action of $\rho_1$ on $X_1$, $X_2$, $X_3$ (and on $X_3^3$ also) is given by
\begin{equation}
\rho_1: X_1\mapsto X_2/(X_1^2-X_1X_2+X_2^2),~ X_2\mapsto X_1/(X_1^2-X_1X_2+X_2^2),~X_3\mapsto -2A/X_3 \label{6-1}
\end{equation}
where $A=g_1g_2g_3^{-1}$ and $g_1$, $g_2$, $g_3$ are the same polynomials defined in Step 6 of Section 4.

Moreover, $\rho_2(X_1)=\rho_1(X_1)$, $\rho_2(X_2)=\rho_1(X_2)$, $\rho_2(X_3)=\sqrt{-1}\rho_1(X_3)$.

Define $Y_1=X_1/X_2$, $Y_2=X_1$, $Y_3=X_3^3/g_1^3$.

The action of $\rho_3$ is given as
\[
\rho_3: X_1\mapsto X_1,~X_2\mapsto X_2,~ X_3\mapsto -\sqrt{-1}
X_3,~ Y_1\mapsto Y_1,~ Y_2\mapsto Y_2,~ Y_3\mapsto \sqrt{-1} Y_3.
\]

We find that
\begin{equation}
k(\sqrt{-1},\sqrt{3})(Y_1,Y_2,Y_3)^{\langle\rho_3\rangle}=k(\sqrt{-3})(Y_1,Y_2,Y_3+\sqrt{-1}Y_3).
\label{6-2}
\end{equation}

For simplicity, call $Y_0=Y_3+\sqrt{-1} Y_3$.
The action of $\rho_1$ is given by
\begin{equation}
\rho_1:\sqrt{-3}\mapsto -\sqrt{-3},~ Y_1\mapsto 1/Y_1,~ Y_2\mapsto Y_1/(Y_2(1-Y_1+Y_1^2)),~ Y_0\mapsto 2B^3/Y_0
\label{6-3}
\end{equation}
where $B$ is defined in Step 6 of Section 4.

The remaining proof is similar to Step 7 of Section 4,
but the ground field is $k(\sqrt{-3})$ in the present situation.

It can be shown that $k(\sqrt{-3})(Y_1,Y_2,Y_0)^{\langle\rho_1\rangle}=k(U_0,U_1,U_2,U_3,U_4)$ where
\begin{align*}
U_0 &= \sqrt{-3}(1-Y_1)/(1+Y_1), ~~ U_1=Y_2+Y_1/(Y_2(1-Y_1+Y_1^2)), \\
U_2 &= \sqrt{-3}[Y_2-Y_1/(Y_2(1-Y_1+Y_1^2))],~~ U_3=Y_0+2B^3/Y_0,~~ U_4=\sqrt{-1}(Y_0-2B^3/Y_0),
\end{align*}
with two relations
\begin{equation}
3U_1^2+U_2^2 = 4(3+U_0^2)/(1-U_0^2), ~~ 3U_3^2+U_4^2=24B^3.
\label{6-4}
\end{equation}

We will simplify the relations in \eqref{6-4}.
We use the identity $(3a^2+b^2)(3c^2+d^2)=3(ad+bc)^2+(bd-3ac)^2$ this time.
Define
\[
V_1=(U_0U_1+U_2)(1-U_0^2)/(6+2U_0^2),~~ V_2=(U_0U_2-3U_1)(1-U_0^2)/(6+2U_0^2).
\]

The first relation becomes
\begin{equation}
3V_1^2+V_2^2=1-U_0^2. \label{6-5}
\end{equation}

Hence we get
\[
3[V_1/(1+V_2)]^2+[U_0/(1+V_2)]^2=(1-V_2)/(1+V_2).
\]

It is not difficult to show that $k(U_0,U_1,U_2)=k(W_1,W_2)$ where $W_1=V_1/(1+V_2)$, $W_2=U_0/(1+V_2)$,
because of Formula \eqref{6-5}.

Now we will simplify the second relation $3U_3^2+U_4^2=24B^3$ in
\eqref{6-1}.

Express $B^3$ in terms of $W_1$, $W_2$.
This relation becomes
\begin{equation}
3U_3^2+U_4^2=3/D^3 \label{6-6}
\end{equation}
where
$D=2[9W_1^2-W_2^2-3(3W_1^2+W_2^2)^2]/[3(3W_1^2+(1-W_2)^2)(3W_1^2+(1+W_2)^2)]$.

Regarding
$(3W_1^2+(1-W_2)^2)(3W_1^2+(1+W_2)^2)=(3a^2+b^2)(3c^2+d^2)$ and
using the identity $(3a^2+b^2)(3c^2+d^2)=3(ad+bc)^2+(bd-3ac)^2$,
we may change the variables by defining
\[
W_3=D^2[2W_1U_4+(3W_1^2+W_2^2-1)U_3],~~ W_4=D^2[6W_1U_3-(3W_1^2+W_2^2-1)U_4].
\]

We get $k(W_1,W_2,U_3,U_4)=k(W_1,W_2,W_3,W_4)$ with a relation
\[
3W_3^2+W_4^2=2[9W_1^2-W_2^2-3(3W_1^2+W_2^2)^2]
\]

Define $w_i=W_i/(3W_1^2+W_2^2)$ for $1\le i\le 4$.
We get $k(W_1,W_2,W_3,W_4)=k(w_1,w_2,\break w_3,w_4)$ with a relation
\begin{equation}
3w_3^2+w_4^2-2(9w_1^2-w_2^2)+6=0. \label{6-7}
\end{equation}

Hence $k(w_1,w_2,w_3,w_4)=k(3w_1-w_2,w_3,w_4)$ is rational over $k$,
because of Formula \eqref{6-7}. Done.

\bigskip

Step 4. Suppose $\pi=\langle\rho_1\rangle$, $\langle\rho_2\rangle$
or $\langle\rho_3\rangle$ for the group $(4,33,7)$.

If $\pi=\langle\rho_3\rangle$, i.e.\ $\sqrt{-3}\in k$, the
rationality was already proved in Formula \eqref{6-2} of the
previous step.

If $\pi=\langle\rho_1\rangle$, i.e.\ $\sqrt{3}\in k$, we get
$k(\sqrt{-1},\sqrt{3})(Y_1,Y_2,Y_3)^\pi=k(\sqrt{-3})(Y_1,Y_2,Y_0)^{\langle\rho_1\rangle}$.
The proof is the same as in Step 3.

The case $\pi=\langle\rho_2\rangle$, i.e.\ $\sqrt{-1}\in k$, is
similar and the proof is omitted.

\bigskip

Step 5.
Now we consider the group $G=\langle\lambda_1,\lambda_2,\sigma,\lambda_3,\lambda_4\rangle$,
i.e.\ the group $(4,33,11)$.
In Step 2, we have shown that $k(\sqrt{-1},\sqrt{3})(z_1,z_2,z_3)^{\langle\lambda_1,\lambda_2,\sigma,\lambda_3\rangle}
=k(\sqrt{-1},\sqrt{3})(X_1,\break X_2,X_3^3)$.
We will show that $k(\sqrt{-1},\sqrt{3})(X_1,X_2,X_3^3)^{\langle\lambda_4,\pi\rangle}$ is rational over $k$.

We modify the definition of $Y_1$, $Y_2$, $Y_3$ in Step 4.
Now define
\[
Y_1=X_1,~~ Y_2=X_2,~~ Y_3=(1+\sqrt{-1})X_3^3/(1+2X_1-X_2+X_1^2-X_1X_2+X_2^2)^3.
\]

The actions of $\lambda_4$, $\rho_1$, $\rho_2$, $\rho_3$ on
$\bm{Q}(\sqrt{-1},\sqrt{3})(Y_1,Y_2,Y_3)$ are given by
\begin{equation}
\begin{aligned}
\lambda_4: Y_1 &\leftrightarrow Y_2,~ Y_3\mapsto -8/(Y_3D^3), \\
\rho_1:Y_1 &\mapsto Y_2/(Y_1^2-Y_1Y_2+Y_2^2), ~ Y_2\mapsto Y_1/(Y_1^2-Y_1Y_2+Y_2^2), \\
Y_3 &\mapsto -8(Y_1^2-Y_1Y_2+Y_2^2)^3/(Y_3D^3), \\ \rho_2:Y_1 &\mapsto Y_2/(Y_1^2-Y_1Y_2+Y_2^2), ~ Y_2\mapsto Y_1/(Y_1^2-Y_1Y_2+Y_2^2), \\
Y_3 &\mapsto -8(Y_1^2-Y_1Y_2+Y_2^2)^3/(Y_3D^3), \\
\rho_3: Y_1 &\mapsto Y_1,~ Y_2\mapsto Y_2,~ Y_3\mapsto Y_3.
\end{aligned} \label{6-8}
\end{equation}
where
$D=1+Y_1+Y_2+2Y_1Y_2+Y_1^3+Y_2^3+Y_1^4-2Y_1^3Y_2+3Y_1^2Y_2^2-2Y_1Y_2^3+Y_2^4$.

Define
\begin{alignat*}{2}
U_0 &= Y_1/Y_2, &~~ U_1 &= Y_1[1+(1/(Y_1^2-Y_1Y_2+Y_2^2))], \\
U_2 &= \sqrt{-3} Y_1[1-(1/(Y_1^2-Y_1Y_2+Y_2^2))], & U_3 &=
Y_3[1+(Y_1^2-Y_1Y_2+Y_2^2)^3].
\end{alignat*}

It is easy to verify that $k(\sqrt{-1},\sqrt{3})(Y_1,Y_2,Y_3)=k(\sqrt{-1},\sqrt{3})(U_0,U_1,U_2,U_3)$ with a relation
\begin{equation}
3U_1^2+U_2^2=12U_0^2/(1-U_0+U_0^2). \label{6-9}
\end{equation}

Define
\[
V_0=2U_0-1,~~V_1=U_1(1-U_0+U_0^2)/U_0,~~V_2=U_2(1-U_0+U_0^2)/(3U_0).
\]

Then $k(\sqrt{-1},\sqrt{3})(U_0,U_1,U_2)=k(\sqrt{-1},\sqrt{3})(V_0,V_1,V_2)$ and the relation \eqref{6-9} becomes
\[
V_0^2-V_1^2-3V_2^2+3=0.
\]

Since $V_0-V_1\in k(\sqrt{-1},\sqrt{3})(V_0+V_1,V_2)$ by the above relation,
we find that $k(\sqrt{-1},\sqrt{3})(V_0,V_1,V_2)=k(\sqrt{-1},\sqrt{3})(V_0+V_1,V_2)$.

Now define
\[
w_1=V_0+V_1,~~ w_2=V_2,~~ w_3=U_3.
\]

Then
$k(\sqrt{-1},\sqrt{3})(Y_1,Y_2,Y_3)=k(\sqrt{-1},\sqrt{3})(U_0,U_1,U_2,U_3)=k(\sqrt{-1},\sqrt{3})(w_1,w_2,w_3)$
and the actions of $\lambda_4, \rho_1, \rho_2, \rho_3$ on $w_1$,
$w_2$, $w_3$ are given by
\begin{equation}
\begin{aligned}
\lambda_4: w_1&\mapsto f_1f_2f_3^{-1},~ w_2\mapsto
4w_1w_2f_3^{-1},~w_3\mapsto -h_1^2h_2^2h_3^2/(w_3w_1^3h_4^3), \\
\rho_1: w_1&\mapsto f_1f_2f_3^{-1},~ w_2\mapsto
4w_1w_2f_3^{-1},~w_3\mapsto -h_1^2h_2^2h_3^2/(w_3w_1^3h_4^3), \\
\rho_2: w_1&\mapsto f_1f_2f_3^{-1},~ w_2\mapsto
4w_1w_2f_3^{-1},~w_3\mapsto -h_1^2h_2^2h_3^2/(w_3w_1^3h_4^3), \\
\rho_3: w_1&\mapsto w_1,~ w_2\mapsto w_2,~ w_3\mapsto w_3,~
\end{aligned}
\end{equation}
where $f_1=3+w_1-3w_2$, $f_2=3+w_1+3w_2$,
$f_3=-3+2w_1+w_1^2+3w_2^2$, $h_1=3+w_1^2-3w_2^2$,
$h_2=3+w_1^2-6w_1w_2-3w_2^2$, $h_3=3+w_1^2+6w_1w_2-3w_2^2$,
$h_4=3+7w_1+5w_1^2+w_1^3-3w_2^2-9w_1w_2^2$.

\bigskip

Step 6. We claim that
$k(\sqrt{-1},\sqrt{3})(w_1,w_2,w_3)^{\langle\lambda_4,
\pi\rangle}=k(w_1,w_2,w_3)^{\langle\lambda_4\rangle}$ no matter
what $\pi$ is. For example, suppose $\pi=\langle\rho_1\rangle$. By
Formula (6.10), the actions of $\lambda_4$ and $\rho_1$ on
$w_1,w_2,w_3$ are the same. Thus
$k(\sqrt{-1},\sqrt{3})(w_1,w_2,w_3)^{\langle\lambda_4,
\rho_1\rangle}=\{ k(\sqrt{-1},\sqrt{3})(w_1, \break
w_2,w_3)^{\langle\lambda_4\rho_1\rangle} \}^{\langle\lambda_4
\rangle}=k(w_1,w_2,w_3)^{\langle\lambda_4\rangle}$. Similarly the
other cases can be verified easily.

\bigskip

Step 7. We will prove that
$k(w_1,w_2,w_3)^{\langle\lambda_4\rangle}$ is $k$-rational.

Note that $\lambda_4(k(w_1,w_2))=k(w_1,w_2)$. We will find an
element $t\in k(w_1,w_2)$ such that $\lambda_4(t)=t$. But we
simplify $w_1$, $w_2$, $w_3$ first.

Define $W_1=w_1+1$, $W_2=w_2$, $W_3=h_1h_2h_3/(w_3w_1^2h_4)$.
Then $k(w_1,w_2,w_3)=k(W_1,W_2,W_3)$ and $\lambda_4$ acts on $W_1$, $W_2$, $W_3$ by
\begin{align*}
\lambda_4:{} & W_1 \mapsto (4W_1+2W_1^2-6W_2^2)/(-4+W_1^2+3W_2^2), \\
& W_2\mapsto(4W_1W_2-4W_2)/(-4+W_1^2+3W_2^2), \\
& W_3\mapsto -4(2W_1^2+W_1^3-9W_1W_2^2+6W_2^2)/[W_3(-4+W_1^2+3W_2^2)].
\end{align*}

We will find an element $t\in k(w_1,w_2)=k(W_1,W_2)$ such that
$\lambda_4(t)=t$. We use a similar trick in \cite[Section
2]{HoK1}: Examine the effects of $\lambda_4$ on $w_1$, $w_2$,
$f_1$, $f_2$, $f_3$. The action $\lambda_4$ acts on these five
polynomials by ``monomial automorphisms" (but these five
polynomial are not algebraically independent). It is not difficult
to find a monomial fixed by $\lambda_4$; for example,
$\lambda_4(f_1/f_2)=f_1/f_2$. Define
\[
t_1=f_1/f_2=(2+W_1-3W_2)/(2+W_1+3W_2).
\]

Then $k(W_1,W_2)=k(t_1,W_1)$ and $\lambda_4(t_1)=t_1$,
$\lambda_4(W_1)=[(-1+2t_1-t_1^2)+W_1(1+4t_1+t_1^2)]/[(-1-4t_1-t_1^2)+W_1(1+t_1+t_1^2)]$.
Note that $\lambda_4(W_1)$ is of the form $(a_1+a_2W_1)/(a_3+a_4W_1)$ where $a_1,a_2,a_3,a_4\in k(t_1)$.
Thus we can bring it into the form $a_0/W'$ for some $W'$ and some $a_0\in k(t_1)$.

Define $E=t_1/(1+t_1+t_1^2)$, $t_2=W_1-3E-1$, $t_3=W_3$.

We find that $k(W_1,W_2,W_3)=k(t_1,t_2,t_3)$ and
\[
\lambda_4: t_1\mapsto t_1,~ t_2\mapsto a/t_2,~ t_3\mapsto [c(t_2+(a/t_2))+d]/t_3
\]
where $a=9E(1+E)$, $c=-12E$, $d=4(1-9E-18E^2)$.
Thus we apply Theorem \ref{t2.3} to conclude that $k(t_1,t_2,t_3)^{\langle\lambda_4\rangle}$ is rational over $k$.
 \qed

\section*{Appendix}

In the first paragraph of Section 3.3 in \cite[page 368]{KY},
the proof of the following theorem is referred to \cite{Ya1},
to which it is not easy to get access.
And therefore we include a proof of this theorem.

\renewcommand{\thetheorem}{A\arabic{theorem}}
\begin{theorem} \label{tA1}
Let $G$ be one of the five groups $(4,31,3)$, $(4,31,4)$, $(4,31,5)$, $(4,31,6)$, $(4,31,7)$ in $GL_4(\bm{Q})$.
Then $\bm{Q}(x_1,x_2,x_3,x_4)^G$ is rational over $\bm{Q}$.
\end{theorem}

\begin{proof}
By \cite[378]{KY}, these groups are isomorphic to $A_5$, $S_5$,
$A_5\times C_2$ or $S_5\times C_2$ as abstract groups where $A_5$
and $S_5$ are the alternating and symmetric group of degree five.
Since $S_5$ has two inequivalent irreducible representations of
dimension four arising from the standard representation and its
tensor product with the unique non-trivial linear character
\cite[pages 27--29]{FH}, it is easy to find these five groups as
follows.

Define $\sigma, \tau_1, \tau_2, \tau_3, \lambda\in GL_4(\bm{Q})$ by
\begin{align*}
\sigma &= \begin{pmatrix} 0 & 0 & 0 & -1 \\ 1 & 0 & 0 & -1 \\ 0 & 1 & 0 & -1 \\ 0 & 0 & 1 & -1 \end{pmatrix},
~~ \tau_1=\left(\hs{-1}\begin{array}{cc@{\vds}cc}
0 & 1 & & \\ 1 & 0 & & \\[-3\jot] \hds{4} \\[-1\jot] & & 1 & 0 \\ & & 0 & 1 \end{array}\hs{-1}\right),
~~ \tau_2=\left(\hs{-1}\begin{array}{cc@{\vds}cc}
0 & -1 & & \\ -1 & 0 & & \\[-3\jot] \hds{4} \\[-1\jot] & & -1 & 0 \\ & & 0 & -1 \end{array}\hs{-1}\right), \\
\tau_3 &= \left(\hs{-1}\begin{array}{ccc@{\vds}c}
0 & 0 & 1 & \\ 1 & 0 & 0 & \\ 0 & 1 & 0 \\[-3\jot] \hds{4} \\[-1\jot] & & & 1 \end{array}\hs{-1}\right),
~~ \lambda=\begin{pmatrix} -1 & & & \\ & -1 & & \\ & & -1 & \\ & & & -1 \end{pmatrix}.
\end{align*}

Note that the matrix $\sigma$ corresponds to the $5$-cycle
$(1,2,3,4,5)\in S_5$, the matrix $\tau_1$ corresponds to the
transposition $(1,2)\in S_5$, the matrix $\tau_3$ corresponds to
the $3$-cycle $(1,2,3)\in S_5$.

By comparing the character tables of these groups, we find that
the group $(4,31,3)$ is conjugate to the group
$\langle\sigma,\tau_3\rangle$ (which is the restriction of the
standard representation to $A_5$), the group $(4,31,4)$ conjugate
to the group $\langle\sigma,\tau_1\rangle$ (which is the standard
representation of $S_5$), the group $(4,31,5)$ conjugate to the
group $\langle\sigma,\tau_2\rangle$, the group
 $(4,31,6)$ conjugate to the group $\langle\sigma,\tau_3,\lambda\rangle$, and the group $(4,31,7)$
conjugate to the group $\langle\sigma,\tau_1,\lambda\rangle$.

Suppose $G$ is any one of the above five groups acting on $\bm{Q}(x_1,x_2,x_3,x_4)$.
By Theorem \ref{t2.1}, we find that $\bm{Q}(x_1,x_2,x_3,x_4)^G=\bm{Q}(x_1/x_4,x_2/x_4,x_3/x_4)^G(x_0)$
where $x_0$ is fixed by all elements of $G$.
Note that $\bm{Q}(x_1/x_4,x_2/x_4,x_3/x_4)$ is the function field of $\bm{P}(V_0)$ where $V_0$ is the
standard representation of $S_5$ (see \cite[p.519]{HK}).
By \cite[Lemma 1, Lemma 5]{HK}, both $\bm{Q}(x_1/x_4,x_2/x_4,x_3/x_4)^{S_5}$ and
$\bm{Q}(x_1/x_4,x_2/x_4,x_3/x_4)^{A_5}$ are $\bm{Q}$-rational.
\end{proof}

The following theorem was due to Oura and Rikuna \cite{OR}.
It seems that \cite{OR} has not been published in some journal.
Thus we include its proof here.

\begin{theorem} \label{tA2}
Let $G$ be a finite subgroup of $GL_3(\bm{Q})$.
Then $\bm{Q}(x_1,x_2,x_3)^G$ is rational over $\bm{Q}$.
\end{theorem}

\begin{proof}
Step 1. We look into the book \cite{BBNWZ}. There are 32 finite
subgroups in $GL_3(\bm{Q})$ up to conjugation. Among them, there
are 22 subgroups in total, which are reducible (i.e.\ suppose $G$
acts on $\bigoplus_{1\le i\le 3} \bm{Q}\cdot x_i$; then without
loss of generality we may assume that $\sigma\cdot x_1,
\sigma\cdot x_2\in\bm{Q}\cdot x_1\oplus \bm{Q}\cdot x_2$ and
$\sigma\cdot x_3\in\bm{Q} \cdot x_3$ for any $\sigma\in G$). The
remaining 10 subgroups are $M$-groups in the sense that, if $G$
acts on $\bigoplus_{1\le i\le 3}\bm{Q}\cdot x_i$, then
$\sigma\cdot x_i\in\bm{Q}\cdot x_{\sigma(i)}$ for all $\sigma\in
G$, for $1\le i\le 3$ (see \cite[page 262]{CR} for details). The
following 10 groups are $M$-groups: $(3,5,i)$, $(3,7,i)$ where
$1\le i\le 5$. The remaining 22 groups are reducible groups. Be
aware that there are more $M$-groups other than the 10 groups
listed above. But these ``extra" groups are reducible groups also.

\bigskip

Step 2.
Suppose $G$ is a reducible group.
We may assume that, for all $\sigma\in G$,
$\sigma\cdot x_1, \sigma\cdot x_2\in\bm{Q}\cdot x_1\oplus \bm{Q}\cdot x_2$, $\sigma\cdot x_3\in \bm{Q}\cdot x_3$.
Thus $\bm{Q}(x_1,x_2,x_3)^G=\bm{Q}(x_1,x_2)^G(x_0)$ by Theorem \ref{t2.1}.
Now $\bm{Q}(x_1,x_2)^G=\bm{Q}(x_1/x_2,x_2)^G=\bm{Q}(x_1/x_2)^G(y_0)$ by Theorem \ref{t2.1} again.
Since $\bm{Q}(x_1/x_2)^G$ is $\bm{Q}$-rational by L\"uroth's Theorem,
we find that $\bm{Q}(x_1,x_2,x_3)^G$ is $\bm{Q}$-rational.

\bigskip

Step 3. Suppose that $G$ is a $M$-group acting on
$\bm{Q}(x_1,x_2,x_3)$. Then
$\bm{Q}(x_1,x_2,x_3)^G=\bm{Q}(x_1/x_3,x_2/x_3,x_3)^G=\bm{Q}(x_1/x_3,x_2/x_3)^G(x_0)$
by Theorem \ref{t2.1}. The action of $G$ on
$\bm{Q}(x_1/x_3,x_2/x_3)$ are monomial actions in $x_1/x_3$,
$x_2/x_3$ (see the last paragraph of Section 1 for the definition
of monomial automorphisms). By Theorem \ref{t2.6},
$\bm{Q}(x_1/x_3,x_2/x_3)^G$ is rational over $\bm{Q}$.
\end{proof}

\bigskip

\bigskip

\bigskip

\renewcommand{\refname}{\centering{References}}


\begin{thebibliography}{Ya2}



\bibitem[AHK]{AHK}
H. Ahmad, M. Hajja and M. Kang,
\textit{Rationality of some projective linear actions},
J. Algebra 228 (2000), 643--658.

\bibitem[BBNWZ]{BBNWZ}
H. Brown, R. B\"ullow, J. Neub\"user, H. Wondratschek and H.
Zassenhaus, \textit{Crystallographic groups of four-dimensional
spaces}, John Wiley, New York, 1978.

\bibitem[CR]{CR}
C. W. Curtis and I. Reiner, \textit{Methods of representation
theory vol. 1}, John Wiley, New York, 1981.

\bibitem[FH]{FH}
W. Fulton and J. Harris, \textit{Representation theory: a first
course}, GTM vol.129, Springer-Verlag, Berlin, 1991.


\bibitem[Ha]{Ha}
M. Hajja,
\textit{Rationality of finite groups of monomial automorphisms of $K(x,y)$},
J. Algebra 109 (1987), 46--51.

\bibitem[HK]{HK}
M. Hajja and M. Kang,
\textit{Some actions of symmetric groups},
J. Algebra 177 (1995), 511--535.

\bibitem[HoK1]{HoK1}
A. Hoshi and M. Kang, \textit{A rationality problem of some
Cremona transformation}, Proc. Japan Academy 84, Ser. A (2008),
133--137.


\bibitem[HoK2]{HoK2}
A. Hoshi and M. Kang, \textit{Twisted symmetric group actions}, to
appear in ``Pacific J. Math.".




\bibitem[Ka]{Ka}
M. Kang,
\textit{Rationality problem of $GL_4$ group actions},
Advances in Math. 181 (2004), 321--352.


\bibitem[Ki]{Ki}
H. Kitayama, \textit{Noether's problem for four-and five-
dimensional linear actions}, to appear in ``J. Algebra".

\bibitem[KY]{KY}
H. Kitayama and A. Yamasaki,
\textit{The rationality problem four four-dimensional linear actions},
J. Math. Kyoto Univ. 49 (2009), 359--380.

\bibitem[Ma]{Ma}
K. Masuda, \textit{On a problem of Chevalley}, Nagoya Math. J. 8
(1955), 59--63.

\bibitem[OR]{OR}
M. Oura and Y. Rikuna, \textit{On three-dimensional linear
Noether's problem}, preprint, 2003.


\bibitem[Pl]{Pl}
B. Plans, \textit{Noether's problem for $GL(2,3)$}, Manuscripta
Math. 124 (2007), 481--487.

\bibitem[Ri]{Ri}
Y. Rikuna, \textit{The existence of generic polynomial for
$SL(2,3)$ over $\bm{Q}$}, preprint.




\bibitem[Sa]{Sa}
D. J. Saltman,
\textit{Generic Galois extensions and problems in field theory},
Advances in Math. 43 (1982), 250--283.

\bibitem[Ya1]{Ya1}
A. Yamasaki,
\textit{Four dimensional linear Noether's problem}, preprint.

\bibitem[Ya2]{Ya2}
A. Yamasaki,
\textit{Negative solutions to three-dimensional monomial Noether problem},
arXiv: 0909.0586v1.


\end{thebibliography}
\end{document}